\documentclass{amsart}
\usepackage{hyperref}
\usepackage{txfonts}
\usepackage{amsfonts}

\newtheorem{theorem}{Theorem}[section]
\newtheorem{lemma}[theorem]{Lemma}
\newtheorem{corollary}[theorem]{Corollary}
\newtheorem{proposition}[theorem]{Proposition}

\theoremstyle{definition}

\newtheorem{example}[theorem]{Example}

\theoremstyle{remark}

\numberwithin{equation}{section}

\begin{document}

\title{integral of distance function on compact Riemannian manifolds}

\author{}
\address{}
\curraddr{}
\email{}
\thanks{}

\author{Jianming Wan}
\address{School of Mathematics, Northwest University, Xi'an 710127, China}
\email{wanj\_m@aliyun.com}
\thanks{The author was supported by National Natural Science Foundation of China
N0.11301416.}


\date{April 26, 2016}

\dedicatory{}


\begin{abstract}\normalsize
In this paper we show that, under some curvature assumptions the integral of distance function on a compact Riemannian manifold is bounded below by the product of diameter, volume and a constant
only depending on the dimension.
\end{abstract}
\maketitle




\section{introduction}
Let $M$ be a compact Riemannian manifold. Let $d(p, q)$ be the distance between points $p$ and $q$. If we fix a point $p\in M$, then we obtain a distance function $d(p,x), x\in M$. It is a continuous function and differentiable almost everywhere. The distance function plays a basic role in Riemannian geometry. In this paper we consider the integral of $d(p,x)$ on $M$. This gives a function
\begin{equation}
f(p)=\int_{M}d(p, x)dv, p\in M.
\end{equation}
Obviously it has an upper bound $d(M)V(M)$. Here $d(M)$ denotes the diameter of $M$ and $V(M)$ is the volume of $M$.

By the mean value theorem, for every $p\in M$ we can find a point $\xi_{p}\in M$ such that $$f(p)=d(p,\xi_{p})V(M).$$
So we can ask a natural question: For any compact Riemannian manifold $M$ of dimension $n$, do we have
\begin{equation}
f(p)\geq c(n)d(M)V(M)
\end{equation}
for all $p\in M$? The $c(n)$ is a positive constant only depending on the dimension $n$. Unfortunately, the answer is negative. In fact we can construct examples such that $\frac{f(p)}{d(M)V(M)}$ is arbitrarily small for some point $p$ (example 2.4). But if we add some curvature conditions, the answer is positive. The first result of this paper is

\begin{theorem}
Let $M$ be an $n$-dimensional compact Riemannian manifold with nonnegatively Ricci curvature. Then $$f(p)>c(n)d(M)V(M)$$
for all $p\in M$. The $c(n)$ can be chosen to equal $(1-\frac{1}{n+1})^{n}\cdot\frac{1}{2^{n+1}(n+1)}$.
\end{theorem}

A natural problem is to determine the sharp value for $c(n)$. Examples in section 2 show that $c(n)$ can achieve $\frac{1}{2}$.

 The Bishop-Gromov's volume comparison plays an essential role in the proof of the theorem.

 In last section we will prove two similar results on complete noncompact Riemannian manifolds.

\section{examples and properties}
Following the triangle inequality, one obviously has
\begin{itemize}
\item $f(p)+f(q)\geq d(p,q)V(M)$,
\item $|f(p)-f(q)|\leq d(p,q)V(M)$.
\end{itemize}
Let $p, q\in M$ satisfy $d(p,q)=d(M)$. One can see that either $f(p)$ or $f(q)$ is greater than or equal to $\frac{1}{2}d(M)V(M)$. So $$\max_{y\in M} f(y)\geq \frac{1}{2}d(M)V(M).$$
We present three examples such that $f(y) \geq \frac{1}{2}d(M)V(M)$ holds for all $y\in M$.

\begin{example}
Let $\gamma$ be a closed curve with length $l$. $d(\gamma)=\frac{l}{2}$ and $V(\gamma)=l$. Then $f=\frac{1}{2}d(\gamma)V(\gamma)$. In fact, if we assume that $\gamma(t)$ has arc length parameter, then
$$f=\int_{0}^{l/2}tdt+\int_{l/2}^{l}(l-t)dt=\frac{l^{2}}{4}=\frac{1}{2}d(\gamma)V(\gamma).$$
\end{example}

\begin{example}
If $M$ is a compact symmetric space, then for any points $p,q\in M$ there exists an isometric mapping $p$ to $q$. Hence $f$ is a constant. Choose $p,q$ such that $d(p, q)=d(M)$. Then $2f=f(p)+f(q)\geq d(p,q)d(M)$.  So we have $$f\geq \frac{1}{2}d(M)V(M).$$  For a special case when $M$ is sphere space form $S^{n}_{k}$ ($k$ is the sectional curvature), let $p\in S^{n}_{k}$ and $q$ is the antipodal point of $p$. Then for any $x\in S^{n}_{k}$,
$d(p, x)+d(q,x)=d(S^{n}_{k})$. Hence we have $$f=\frac{1}{2}d(S^{n}_{k})V(S^{n}_{k}).$$
\end{example}

\begin{example}
Let $T^{2}$ be the flat 2-torus of area 1. Then $$f=\int_{0}^{\frac{1}{2}}r\cdot 2\pi rdr+4\int_{\frac{1}{2}}^{\frac{\sqrt{2}}{2}}r\cdot(\frac{\pi}{2}-2\arccos\frac{1}{2r})rdr$$ $$=\frac{1}{6}(\sqrt{2}+\ln(\sqrt{2}+1))\doteq 0.3826.$$
$f>\frac{1}{2}diam(T^{2})V(T^{2})=\frac{\sqrt{2}}{4}\doteq0.3535$.
\end{example}

However the following example shows that $\frac{f(y)}{d(M)V(M)}$ can achieves every value in $(0,1)$.

\begin{example}
Let $M_{1}=\{(x,y,z)|x^{2}+y^{2}+z^{2}=1, -1+\epsilon\leq z \leq1\}$ and $M_{2}=\{(x,y,z)|x^{2}+y^{2}=2\epsilon-\epsilon^{2}, -1+\varepsilon-L \leq z\leq-1+\varepsilon\}\cup \{(x,y,z)|x^{2}+y^{2}\leq2\epsilon-\epsilon^{2},z=-1+\varepsilon-L\}$. Let $M=M_{1}\cup M_{2}$ be $M_{1}$ glued to $M_{2}$.
We write $C=2\pi \sqrt{2\epsilon-\epsilon^{2}}$. Let $p=(0,0,1)$ and $q=(0,0,-1+\varepsilon-L)$.
When $C$ is very small,
\begin{eqnarray*}
f(p)& \doteq & \int_{S^{2}_{1}}d(p,x)dS+\int^{L}_{0}(\pi+t)Cdt\\
& = & \frac{1}{2}\pi V(S^{2}_{1})+C(\pi L+\frac{L^{2}}{2})
\end{eqnarray*}
and
$$d(M)V(M)\doteq(L+\pi)(V(S^{2}_{1})+LC).$$
Set $C=\frac{1}{L^{3}}$. We can see that $\frac{f(p)}{d(M)V(M)}\rightarrow 0$ as $L\rightarrow \infty$. We also have  $\frac{f(q)}{d(M)V(M)}\rightarrow 1$ as $L\rightarrow \infty$.
\end{example}

The following proposition is a consequence of Bishop-Gromov volume comparison.
\begin{proposition}
If $Ric_{M}\geq (n-1)k>0$, then $f\leq \frac{1}{2}d(S^{n}_{k})V(S^{n}_{k})$. The equality holds if and only if $M$ is isometric to $S^{n}_{k}$.
\end{proposition}

\begin{proof}
By the Fubini theorem and Bishop-Gromov volume comparison,
\begin{eqnarray*}
f(p)& = & \int_{M}d(p, x)dv = \int_{-\infty}^{+\infty}d\lambda \int_{S_{\lambda}}d(p, x)dv_{\lambda}\\
& = &\int_{0}^{d(M)}d\lambda\int_{S_{\lambda}}\lambda dv_{\lambda}  = \int_{0}^{d(M)}\lambda V(S_{\lambda})d\lambda\\
& \leq &\int_{0}^{d(S^{n}_{k})}\lambda V(S_{k\lambda})d\lambda\\
& = &\frac{1}{2}d(S^{n}_{k})V(S^{n}_{k}).
\end{eqnarray*}
The $S_{\lambda}$ denotes the sphere center at $p$ with radius $\lambda$ and $V(S_{\lambda})$ is the induced volume of $S_{\lambda}$. If the equality holds, we
 have $d(M)=d(S^{n}_{k})$ and $V(S_{\lambda})=V(S_{k\lambda})$. So $M$ must be isometric to $S^{n}_{k}$.
\end{proof}

\section{a proof of theorem 1.1}
Let $B_{p}(r)$ (respectively $B_{o}(r)$ ) denote the ball center at $p$ of radius $r$ in $M$ (respectively ball center at origin of radius $r$ in $\mathbb{R}^{n}$). The $V_{p}(r)$ (respectively $V_{o}(r)$) denotes the volume of $B_{p}(r)$ (respectively $B_{o}(r)$).

\begin{lemma}
For any $p\in M$, we have $\max_{x\in M}d(p,x) \geq \frac{1}{2}d(M)$.
\end{lemma}

\begin{proof}
If on the contrary, $\max_{x\in M}d(p,x) < \frac{1}{2}d(M)$. Choosing $q_{1}, q_{2}\in M$ such that $d(q_{1}, q_{2})=d(M)$, thus
$$d(M)\leq d(p, q_{1})+d(p, q_{2})< \frac{1}{2}d(M)+\frac{1}{2}d(M)=d(M).$$ This is a contradiction.
\end{proof}

Because the Ricci curvature of $M$ is nonnegative. The Bishop-Gromov's volume comparison implies that
$$\frac{V_{p}(r)}{V_{o}(r)}\geq \frac{V_{p}(R)}{V_{o}(R)}$$
for $r\leq R$. Hence $$V_{p}(r) \geq \frac{V_{o}(r)}{V_{o}(R)}V_{p}(R)=\frac{r^{n}}{R^{n}}V_{p}(R).$$
Let $R\rightarrow d=d(M)$. We get $$V_{p}(r) \geq \frac{r^{n}}{d^{n}}V(M)$$ for all $r\leq d$.

By the lemma, for any $p\in M$, we can choose $q\in M$ such that $d(p,q)\geq\frac{1}{2}d(M)$. Thus
\begin{eqnarray*}
f(p)& = & \int_{B_{p}(\frac{1}{2}d-r)}d(p,x)dv+\int_{M\setminus B_{p}(\frac{1}{2}d-r)}d(p,x)dv\\
& > & \int_{M\setminus B_{p}(\frac{1}{2}d-r)}d(p,x)dv\\
& > & (\frac{1}{2}d-r) V(M\setminus B_{p}(\frac{1}{2}d-r))\\
& > & (\frac{1}{2}d-r) V_{q}(r)\\
& \geq & V(M)(\frac{1}{2}d-r)\frac{r^{n}}{d^{n}}.
\end{eqnarray*}
Let $g(r)=(\frac{1}{2}d-r)\frac{r^{n}}{d^{n}}, 0\leq r\leq \frac{1}{2}d$. When $$g^{'}(r)=\frac{1}{d^{n}}[\frac{nd}{2}r^{n-1}-(n+1)r^{n}]=0,$$
$r=\frac{n}{2(n+1)}d$,  $g(r)$ achieves the maximal value $(1-\frac{1}{n+1})^{n}\cdot\frac{1}{2^{n+1}(n+1)}d$.
Hence we get $$f(p)>(1-\frac{1}{n+1})^{n}\cdot\frac{1}{2^{n+1}(n+1)}d(M)V(M).$$

\section{noncompact analogues of theorem 1.1}
In this section we consider the noncompact version of theorem 1.1. Let $M$ be a complete noncompact Riemannian manifold of dimension $n$. For a point $p\in M$ and  $d>0$, we write $$f(p,d)=\int_{B_{p}(d)}d(p,x)dv$$ Then we have

\begin{theorem}
 If $M$ is a Cartan-Hadamard manifold, then $$f(p,d)> \frac{n}{n+1}\cdot\frac{1}{\sqrt[n]{n+1}}dV_{p}(d),$$
for any $p\in M$ and all $d>0$.
\end{theorem}

Note that $\frac{n}{n+1}\cdot\frac{1}{\sqrt[n]{n+1}}$ tends to 1 as $n$ goes to $+\infty$. On the other hand, we always have $f(p,d)<dV_{p}(d)$. So
theorem 4.1 is more or less surprise.

\begin{proof}

Since the sectional curvature of $M$ is nonpositive and $M$ has no cut point. By the Bishop-Gromov's volume comparison (c.f. \cite{[GHL]}  page 169), one has
$$\frac{V_{p}(r)}{V_{o}(r)}\leq \frac{V_{p}(d)}{V_{o}(d)}$$
for $r\leq d$. Hence $$V_{p}(r) \leq \frac{V_{o}(r)}{V_{o}(d)}V_{p}(d)=\frac{r^{n}}{d^{n}}V_{p}(d).$$

We estimate the lower bound of $f$.
\begin{eqnarray*}
f(p,d)& = & \int_{B_{p}(r)}d(p,x)dv+\int_{B_{p}(d)\setminus B_{p}(r)}d(p,x)dv\\
& > & \int_{B_{p}(d)\setminus B_{p}(r)}d(p,x)dv\\
& > & r V(B_{p}(d)\setminus B_{p}(r))=r(V_{p}(d)-V_{p}(r))\\
& \geq & V_{p}(d)r(1-\frac{r^{n}}{d^{n}}).
\end{eqnarray*}
Let $g(r)=r(1-\frac{r^{n}}{d^{n}})$, $0\leq r \leq d$. When $$g^{'}(r)=1-(n+1)\frac{r^{n}}{d^{n}}=0, $$ $r=\frac{d}{\sqrt[n]{n+1}}$, $g(r)$ achieves the maximal value $\frac{n}{n+1}\cdot\frac{1}{\sqrt[n]{n+1}}\cdot d$. So we have $$f(p,d)> \frac{n}{n+1}\cdot\frac{1}{\sqrt[n]{n+1}}dV_{p}(d).$$

\end{proof}

\begin{theorem}
If the Ricci curvature of $M$ is nonnegative, then $$f(p,d)> c(n)dV_{p}(d),$$
for any $p\in M$ and all $d>0$. The concrete value of $c(n)$ is given in the following proof.
\end{theorem}

The constant $c(n)$ is different to the one in Theorem 1.1.

\begin{proof}
 Let $0\leq t\leq\frac{d}{2}$. $q$ is a point satisfying $d(p,q)=d-t$. Then
\begin{eqnarray*}
f(p,d)& = & \int_{B_{p}(d-2t)}d(p,x)dv+\int_{B_{p}(d)\setminus B_{p}(d-2t)}d(p,x)dv\\
& > & \int_{B_{p}(d)\setminus B_{p}(d-2t)}d(p,x)dv\\
& > & (d-2t) V(B_{p}(d)\setminus B_{p}(d-2t))\\
& > & (d-2t) V_{q}(t)\\
& \geq & (d-2t)\frac{t^{n}}{(2d-t)^{n}}V_{q}(2d-t)\\
& > & (d-2t)\frac{t^{n}}{(2d-t)^{n}}V_{p}(d).
\end{eqnarray*}
Since $B_{q}(t)\subset B_{p}(d)\setminus B_{p}(d-2t)$, the third ``$>$'' holds. The last ``$>$'' follows from $B_{q}(2d-t)\supset B_{p}(d)$. Let $g(t)=(d-2t)\frac{t^{n}}{(2d-t)^{n}},0\leq t\leq\frac{d}{2}$. When
$$g^{'}(t)=-2(\frac{t}{2d-t})^{n}+(d-2t)n(\frac{t}{2d-t})^{n-1}\frac{2d}{(2d-t)^{2}}=0,$$
namely, $$t^{2}-2d(n+1)t+nd^{2}=0,$$
$t=(n+1-\sqrt{n^{2}+n+1})d$. $g(t)$ achieves the maximal value $\frac{3}{2\sqrt{n^{2}+n+1}+2n+1}(\frac{n}{n+2+2\sqrt{n^{2}+n+1}})^{n}d$. Choose
$$c(n)=\frac{3}{2\sqrt{n^{2}+n+1}+2n+1}(\frac{n}{n+2+2\sqrt{n^{2}+n+1}})^{n}.$$ We obtain $$f(p,d)\geq c(n)dV_{p}(d).$$

\end{proof}

Recall a well-known theorem of Calabi and Yau \cite{[Y]}: Let $M$ be a complete noncompact Riemannian manifold with nonnegative Ricci curvature. For any $p\in M$, we have $V_{p}(r)\geq c(n, p)r$. Consequently $M$ has infinite volume. Combining this result with Theorem 4.2, we obtain

\begin{corollary}
Let $M$ be a complete noncompact Riemannian manifold with nonnegative Ricci curvature. Then
$$\lim_{r\rightarrow \infty}\int_{B_{p}(r)}\frac{d(p,x)}{r}dv=+\infty,$$
for all $p\in M$.
\end{corollary}

\bibliographystyle{amsplain}

\end{document}